\documentclass[11pt]{article}

\usepackage[a4paper,margin=1in]{geometry}
\usepackage{amsmath,amssymb,amsthm,mathtools}
\usepackage{enumitem}
\usepackage{microtype}
\usepackage[colorlinks=true,linkcolor=blue,citecolor=blue,urlcolor=blue]{hyperref}
\usepackage{authblk}
\numberwithin{equation}{section}

\newtheorem{theorem}{Theorem}[section]
\newtheorem{lemma}[theorem]{Lemma}
\newtheorem{proposition}[theorem]{Proposition}
\newtheorem{corollary}[theorem]{Corollary}
\newtheorem{remark}[theorem]{Remark}

\newcommand{\R}{\mathbb{R}}
\newcommand{\eps}{\varepsilon}

\newcommand{\supp}{\operatorname{supp}}

\title{New sharp lifespan estimates for a semilinear heat equation\\
with zero-mass sign-changing initial data}
\author[1,2]{Berikbol Torebek}
\affil[1]{\small Institute of
Mathematics and Mathematical Modeling, Almaty, Kazakhstan}
\affil[2]{\small SDU University, Kaskelen, Kazakhstan}
\affil[ ]{\texttt{torebek@math.kz, berikbol.torebek@sdu.edu.kz}}
\date{}

\begin{document}

\maketitle

\begin{abstract}
We consider the Cauchy problem
$$
u_t-\Delta u=|u|^p,\qquad (t,x)\in(0,T)\times\R^n,
\quad
u(0,x)=\eps u_0(x),\qquad x\in\R^n,
$$
with small sign-changing initial data having zero total mass. We assume
$$
u_0\in L^1(\R^n)\cap L^\infty(\R^n),\quad
|x|u_0\in L^1(\R^n),\quad
\int_{\R^n}u_0(x)\,dx=0,
$$
and, for the sharp subcritical upper bounds, that
$
\int_{\R^n}x\,u_0(x)\,dx\neq0.
$
Let \(T_\eps\) denote the maximal lifespan. We identify a new transition exponent
$
p_m=1+\frac1{n+1}
$
inside the Fujita range \(1<p\le p_F:=1+2/n\), and establish the sharp lifespan estimates
$$
T_\eps\asymp
\begin{cases}
\eps^{-\frac{2(p-1)}{2-(n+1)(p-1)}},
&1<p<p_m,\\[2mm]
\eps^{-\frac2{n+1}}
\bigl(\log\frac1\eps\bigr)^{-\frac2{n+2}},
&p=p_m,\\[2mm]
\eps^{-p\left(\frac{1}{p-1}-\frac{n}{2}\right)^{-1}},
&p_m<p<p_F,\\[2mm]
\exp\!\left(\eps^{-p(p-1)}\right), & p=p_F.
\end{cases}
$$
These estimates reveal a lifespan phenomenon that is different from the classical Lee--Ni law for initial data with positive mass. In the zero-mass
setting, the leading linear contribution is dipole-like, while the nonlinear source subsequently generates a positive mass of size \(O(\eps^p)\). Their
competition produces the additional threshold \(p_m\), the logarithmic correction at \(p=p_m\), and, at \(p=p_F\), a substantially longer critical
lifespan with exponent \(p_F(p_F-1)\) instead of the classical
Lee--Ni exponent \(p_F-1\).

The upper estimates are obtained by backward-Gaussian and critical
scale-ODE test-function arguments, whereas the lower estimates follow from a
unified \(L^1\)--\(L^\infty\) bootstrap preserving the zero-mass cancellation
of the linear flow and controlling the mass generated by the nonlinear
source. Thus, although \(p_F\) remains the Fujita critical exponent, zero
initial mass creates a new quantitative lifespan regime below and at the
critical exponent.
\end{abstract}

\noindent\textbf{Keywords.}
Semilinear heat equation; Fujita exponent; sign-changing solution; zero initial moment;
lifespan; first moment.

\medskip
\noindent\textbf{MSC 2020.}
35K58, 35B44, 35B33, 35C06.
\tableofcontents
\section{Introduction and main results}
\subsection{Historical background}
We study the Cauchy problem
\begin{equation}
\begin{cases}
u_t-\Delta u=|u|^p,
& (t,x)\in(0,T)\times\R^n,\\
u(0,x)=\eps u_0(x),
&x\in\R^n,
\end{cases}
\label{eq:problem}
\end{equation}
where $n\ge1$, $p>1$, $\eps>0$ is a small parameter, and the initial datum
$u_0$ is allowed to change sign.

For nonnegative initial data, the classical theory originates with Fujita's seminal 1966 work \cite{Fujita1966}. He identified the critical exponent
$$
p_F=1+\frac{2}{n},
$$
which separates the blow-up and small-data global-existence regimes for the semilinear heat equation. More precisely, every nontrivial nonnegative solution blows up in finite time when
$1<p<p_F,$
whereas for $p>p_F$ global solutions exist for sufficiently small initial data.

The critical case $p=p_F$, which was left open in Fujita's original work,
was subsequently settled by Hayakawa \cite{Hayakawa1973},  and Sugitani
\cite{Sugitani1975}. They showed
that the critical exponent belongs to the nonexistence regime. For further details on the classical Fujita theory and related
developments, see also Weissler \cite{Weissler1981}, Levine
\cite{Levine1990}, and the monograph of Quittner and Souplet
\cite{QuittnerSouplet2007}.

Lifespan estimates and the influence of the spatial decay of the initial datum were studied systematically by Lee and Ni \cite{LeeNi1992}; related lifespan questions were also considered by Pinsky \cite{Pinsky1998}. In particular, for small nonnegative initial data of the form
$
u(0,x)=\varepsilon u_0(x),
\, 0<\varepsilon\ll1,
$
with sufficiently localized \(u_0\), the maximal existence time \(T_\varepsilon\) satisfies the classical subcritical estimate
$$
T_\varepsilon
\asymp
\varepsilon^{-\left(\frac{1}{\,p-1\,}-\frac{n}{\,2\,}\right)^{-1}},
\qquad
1<p<p_F,
$$
while at the critical case
$
p=p_F=1+\frac2n,
$
one has the exponential lifespan scale
$$
T_\varepsilon
\asymp
\exp\left(\varepsilon^{-(p-1)}\right).
$$
Lee and Ni also showed that slower algebraic decay of the initial datum can modify these lifespan laws and leads to additional critical decay thresholds. It is also worth noting that the Lee--Ni lifespan estimates have been extended in various directions to more general parabolic equations and systems; (see, for instance, \cite{{FujiwaraIkedaWakasugi2020, IkedaSobajima2019, TayachiWeissler2023, Tayachi2024}}). Recently, Tobakhanov and the present author \cite{TobakhanovTorebekBLMS} extended the Lee--Ni lifespan estimates to the higher-order parabolic equation
$$
u_t+(-\Delta)^m u=|u|^p,
$$
under the assumption
$
\int_{\mathbb R^n}u_0(x)\,dx>0.
$
In the particular case \(m=1\), their result recovers the classical Lee--Ni lifespan estimate for the semilinear heat equation and, moreover, extends it to sign-changing initial data with positive total mass.

The situation is subtler for sign-changing data. Zhang \cite{Zhang2002} studied the equation with the positive source $|u|^p$ and showed that positive total mass forces non-global behavior in the Fujita range, under the assumptions appearing in that work. Pinsky \cite{Pinsky2005} then investigated the zero-mean case and proved, in particular, that for $p<p_F$ nontrivial zero-mass data do not generate global classical solutions in the natural class considered there. At the critical exponent, Pinsky obtained the corresponding non-global result under an additional
Gaussian decay assumption on the negative part of the initial datum. These results show that the cancellation
\[
\int_{\R^n}u_0(x)\,dx=0
\]
does not prevent finite-time breakdown for the equation with source $|u|^p$.

The aim of the present paper is quantitative. We ask how the cancellation
of the zeroth moment changes the lifespan when the first moment is nonzero.
Our basic assumptions are
\begin{equation}
u_0\in L^1(\R^n)\cap L^\infty(\R^n),\qquad
|x|u_0\in L^1(\R^n),
\label{eq:data-space}
\end{equation}
together with
\begin{equation}
\int_{\R^n}u_0(x)\,dx=0,
\qquad
m:=\int_{\R^n}x\,u_0(x)\,dx\ne0.
\label{eq:moments}
\end{equation}
For convenience in the classical-solution formulation, one may in addition
assume $u_0$ to be bounded and H\"older continuous. The estimates below depend
only on the mild formulation and therefore extend, by standard approximation,
to the usual bounded mild solution class.

The cancellation in \eqref{eq:moments} removes the Gaussian mass term from the
large-time linear expansion. Instead,
\begin{equation*}
e^{t\Delta}u_0(x)
=
-m\cdot\nabla G(t,x)
+o\bigl(t^{-(n+1)/2}\bigr)
\end{equation*}
on the self-similar scale $x=\sqrt t\,\xi$, where
\[
G(t,x)=(4\pi t)^{-n/2}e^{-|x|^2/(4t)}
\]
is the heat kernel. Thus the initial linear profile decays like a dipole,
with the effective decay rate $t^{-(n+1)/2}$ rather than $t^{-n/2}$.
Moment expansions of this type are classical; see, for example,
Duoandikoetxea and Zuazua \cite{DuoZuazua1992}.

A second mechanism is specific to the nonlinearity $|u|^p$. Since the source
is nonnegative, the mild formula gives
\[
u(t)\ge \eps e^{t\Delta}u_0.
\]
Moreover, whenever the solution is integrable,
\[
\frac{d}{dt}\int_{\R^n}u(t,x)\,dx
=
\int_{\R^n}|u(t,x)|^p\,dx\ge0.
\]
Thus a positive zeroth-order mass is generated immediately even though the
initial mass vanishes. On the dipole scale,
\[
\int_{\R^n}\bigl(e^{t\Delta}u_0\bigr)_+^p\,dx
\asymp
t^{-\frac{p(n+1)-n}{2}},
\]
at least from below on a fixed self-similar cone. The time integral of this
quantity changes character at
\[
\frac{p(n+1)-n}{2}=1,
\]
that is, $p_m:=1+\frac1{n+1}.$
This is the transition exponent in our estimates.

It is worth emphasizing that $p_m$ also appears in a different context for the
odd nonlinearity $|u|^{p-1}u$. Ghoul \cite{Ghoul2011,Ghoul2012} studied
zero-mean and first-moment effects for that equation. The mechanism in
\eqref{eq:problem} is different because $|u|^p$ has a fixed positive sign and
therefore creates positive mass. Recent lifespan results for the odd nonlinear
heat equation, including sign-changing and singular data, were developed
further by Tayachi and Weissler \cite{TayachiWeissler2023}. We use these works
only as context; the proof below is adapted to the positive-source structure
of \eqref{eq:problem}.

Our main result establishes sharp two-sided lifespan estimates throughout the Fujita subcritical and critical ranges. In the strict subcritical case \(1<p<p_F\), the upper estimates are obtained by means of a backward Gaussian functional, while the matching lower estimates follow from a nonlinear bootstrap argument that preserves the cancellation of the initial mass. At the Fujita critical exponent \(p=p_F\), the same cancellation-sensitive bootstrap yields the exponential lower bound, whereas the matching upper bound is derived from a sign-changing critical scale-ODE test-function argument. Consequently, the lifespan is determined sharply, up to multiplicative constants in the power-law regimes and multiplicative constants in the exponent at the critical endpoint.

\subsection{Main results and notation}

For $f\in L^1(\R^n)$, write
\[
\|f\|_{L^1_1}
:=
\int_{\R^n}(1+|x|)|f(x)|\,dx.
\]
Let $T_\eps\in(0,\infty]$ be the maximal existence time of the bounded mild
solution of \eqref{eq:problem}. Standard semilinear heat theory gives local
existence, uniqueness and the usual $L^\infty$ blow-up alternative for bounded
initial data; see, for example,
\cite{Weissler1979,Weissler1980,Pinsky2005,QuittnerSouplet2007}.

Throughout the paper we use
\[
p_F:=1+\frac2n,
\qquad
p_m:=1+\frac1{n+1}.
\]
For $1<p<p_F$, define
\begin{equation}
\mathcal L_p(\eps):=
\begin{cases}
\displaystyle
\eps^{-\frac{2(p-1)}{2-(n+1)(p-1)}},
&1<p<p_m,
\\[3mm]
\displaystyle
\eps^{-\frac2{n+1}}
\left(\log\frac1\eps\right)^{-\frac2{n+2}},
&p=p_m,
\\[3mm]
\displaystyle
\eps^{-\frac{2p(p-1)}{2-n(p-1)}},
&p_m<p<p_F.
\end{cases}
\label{eq:main-bound}
\end{equation}

\begin{theorem}[Sharp lifespan up to the Fujita exponent]
\label{thm:main}
Assume \eqref{eq:data-space}--\eqref{eq:moments} and $1<p\le p_F$.
There exist constants $\eps_0,c,C>0$, depending only on $n,p$ and $u_0$,
such that the following assertions hold for $0<\eps<\eps_0$.

\begin{enumerate}[label=\textup{(\roman*)}]
\item If $1<p<p_F$, then
\begin{equation}
c\,\mathcal L_p(\eps)
\le T_\eps\le
C\,\mathcal L_p(\eps).
\label{eq:sharp-subcritical}
\end{equation}

\item If $p=p_F=1+2/n$, then
\begin{equation}
\exp\!\left(c\eps^{-p(p-1)}\right)
\le T_\eps\le
\exp\!\left(C\eps^{-p(p-1)}\right).
\label{eq:sharp-critical}
\end{equation}
\end{enumerate}
In particular, $T_\eps<\infty$ for every sufficiently small $\eps>0$ when
$1<p\le p_F$.
\end{theorem}

\begin{remark}[Role of the nonzero first moment]
The assumption $m\ne0$ is essential for the matching dipole upper bound in the
range $1<p\le p_m$. The lower bounds, the generated-mass upper bound for
$p_m<p<p_F$, and both critical endpoint estimates use only the zero-mass and
finite-first-absolute-moment assumptions together with $u_0\not\equiv0$.
Thus the endpoint assertion in Theorem \ref{thm:main} remains valid without
assuming $m\ne0$.
\end{remark}

\begin{remark}[Interpretation and comparison of the lifespan scales]
The three subcritical regimes in \eqref{eq:main-bound} reflect two competing
mechanisms. For \(1<p<p_m\), the lifespan is governed by the dipole part of the
linear solution, whereas for \(p_m<p<p_F\) it is determined by the positive
zeroth moment generated by the source, whose effective size is of order
\(\eps^p\). At \(p=p_m\), these mechanisms balance and produce the logarithmic
correction. At the Fujita exponent \(p=p_F\), the generated mass leads to the
critical exponential scale
$$
\log T_\eps\asymp \eps^{-p_F(p_F-1)}.
$$
It is useful to compare these estimates with the classical Lee--Ni lifespan
law \cite{LeeNi1992} and with the \(m=1\) case of the result of \cite{TobakhanovTorebekBLMS}. For nonnegative data, or more generally
sign-changing data with positive total mass, the effective amplitude is
\(\eps\); in the present zero-mass problem, the nonlinear evolution generates
an effective mass of order \(\eps^p\) in the mass-dominated regime.
\end{remark}
A detailed comparison of the results discussed above is provided in the following tables (see Tables \ref{tab:subcritical-comparison} and \ref{tab:critical-comparison}).

\begin{table}[ht] \centering \caption{Comparison of small-data lifespan estimates in the Fujita-subcritical range.} \label{tab:subcritical-comparison} \small\renewcommand{\arraystretch}{1.3} \begin{tabular}{p{3.0cm}|p{4.4cm}|p{2.7cm}|p{4.2cm}} \hline Result & Assumptions on $u_0$ & Range of $p$ & Lifespan \\ \hline Lee--Ni \cite{LeeNi1992} & $u_0\ge0$, $u_0\not\equiv0$ & $1<p<p_F$ & $\displaystyle T_\eps\asymp \eps^{-\frac{2(p-1)}{2-n(p-1)}} $ \\[3mm] Tobakhanov--Torebek \cite{TobakhanovTorebekBLMS}, $m=1$ & $\int_{\R^n}u_0>0$ & $1<p<p_F$ & $\displaystyle T_\eps\asymp \eps^{-\frac{2(p-1)}{2-n(p-1)}} $ \\[3mm] Present work & $\int_{\R^n}u_0=0,$ $\int_{\R^n}x\,u_0\neq0$ & $1<p<p_m$ & $\displaystyle T_\eps\asymp \eps^{-\frac{2(p-1)} {2-(n+1)(p-1)}} $ \\[3mm] Present work & same assumptions & $p=p_m$ & $\displaystyle T_\eps\asymp \eps^{-\frac{2}{n+1}} \bigl(\log(1/\eps)\bigr)^{-\frac{2}{n+2}} $ \\[3mm] Present work & same assumptions & $p_m<p<p_F$ & $\displaystyle T_\eps\asymp \eps^{-\frac{2p(p-1)} {2-n(p-1)}} $ \\ \hline \end{tabular} \end{table}

\begin{table}[ht] \centering \caption{Comparison of small-data lifespan estimates at the Fujita critical exponent.} \label{tab:critical-comparison}\small \renewcommand{\arraystretch}{1.3} \begin{tabular}{p{3.2cm}|p{4.8cm}|p{2.4cm}|p{4.0cm}} \hline Result & Assumptions on $u_0$ & Case of $p$ & Critical lifespan \\ \hline Lee--Ni \cite{LeeNi1992} & $u_0\ge0$, $u_0\not\equiv0$ & $p=p_F$ & $\displaystyle \log T_\eps \asymp \eps^{-(p_F-1)} $ \\[3mm] Tobakhanov--Torebek \cite{TobakhanovTorebekBLMS}, $m=1$ & $\int_{\R^n}u_0>0$ & $p=p_F$ & $\displaystyle \log T_\eps \asymp \eps^{-(p_F-1)} $ \\[3mm] Present work & $\int_{\R^n}u_0=0,$ $\int_{\R^n}x\,u_0\neq0$ & $p=p_F$ & $\displaystyle \log T_\eps \asymp \eps^{-p_F(p_F-1)} $ \\ \hline \end{tabular} \end{table}

\begin{remark}[Comparison with Pinsky's endpoint assumption]
Pinsky \cite{Pinsky2005} proved qualitative non-globality at $p=p_F$ for
zero-mean data under an additional Gaussian decay assumption on the negative
part. Our quantitative critical upper bound uses instead the finite absolute
first moment $|x|u_0\in L^1$. The two hypotheses play different roles: the
present moment condition supplies a quantitative lower bound for the pairing
of the time-shifted datum with large critical cutoffs.
\end{remark}

\subsection{The supercritical and negative-mass regimes}
\label{subsec:supercritical-negative-mass}

The preceding results describe the lifespan for zero-mass data in the
Fujita range \(1<p\le p_F\). For completeness, we briefly discuss what
happens outside this setting. Two cases are particularly relevant:
the supercritical range \(p>p_F\) with zero initial mass, and the case of
strictly negative initial mass.

\paragraph{The supercritical zero-mass case.}
Assume
\[
p>p_F=1+\frac2n,
\qquad
\int_{\R^n}u_0(x)\,dx=0,
\]
together with
\[
u_0\in L^1(\R^n)\cap L^\infty(\R^n),
\qquad
|x|u_0\in L^1(\R^n).
\]
Then the cancellation-sensitive bootstrap used in the proof of the lower
lifespan estimates extends globally for sufficiently small \(\eps>0\).
Indeed, if
\[
A(t):=\int_0^t\|u(s)\|_p^p\,ds,
\]
the master estimate has the form
\[
A'(t)
\le
C(1+t)^{-\frac n2(p-1)}
\left[
\eps(1+t)^{-1/2}+A(t)
\right]^p.
\]
Since
\[
\frac n2(p-1)>1
\qquad\text{when }p>p_F,
\]
the time weight is integrable. A bootstrap of the form
\[
A(t)\le K\eps^p
\]
therefore closes uniformly on \([0,\infty)\) for sufficiently small
\(\eps\). Consequently,
\[
T_\eps=\infty.
\]

Thus, in the supercritical range, the zero-mass condition does not lead to a
finite lifespan for sufficiently small initial data. Nevertheless, the
nonlinearity still produces a new asymptotic effect. Since
\[
\frac{d}{dt}\int_{\R^n}u(t,x)\,dx
=
\int_{\R^n}|u(t,x)|^p\,dx\ge0,
\]
the initially vanishing mass becomes positive for every \(t>0\), and
\[
M_\eps
:=
\lim_{t\to\infty}\int_{\R^n}u(t,x)\,dx
=
\int_0^\infty\int_{\R^n}|u(t,x)|^p\,dx\,dt.
\]
The same positivity argument as above yields
\[
M_\eps\asymp\eps^p.
\]
Hence the solution eventually develops a positive Gaussian component of
size \(O(\eps^p)\). In particular, although the linear part is initially
dipole-like because the zeroth moment vanishes, the large-time behavior is
expected to be governed by the nonlinearly generated mass:
\[
u(t,x)\sim M_\eps G(t,x),
\qquad t\to\infty.
\]
Thus the role of the zero-mass condition changes at \(p=p_F\): below and at
\(p_F\) it modifies the lifespan, whereas above \(p_F\) it modifies the
large-time asymptotic profile.

\paragraph{The case of negative initial mass.}
The situation is fundamentally different when
\[
\int_{\R^n}u_0(x)\,dx<0.
\]
In this case there is no classification based solely on the exponent \(p\)
and the sign of the initial mass. Zhang \cite{Zhang2002} showed that in the
Fujita range positive total mass leads to non-global behavior for
sign-changing solutions. Pinsky \cite{Pinsky2005} subsequently studied the
zero- and negative-mass cases. In particular, he proved that for every
$p>1$
and every prescribed number
$\ell<0,$
there exist sign-changing initial data \(u_0\) satisfying
\[
\int_{\R^n}u_0(x)\,dx=\ell
\]
for which the corresponding solution is global, and there also exist initial
data with the same total mass \(\ell\) for which the corresponding solution
is non-global.

Therefore, in contrast with the cases
\[
\int_{\R^n}u_0>0
\qquad\text{and}\qquad
\int_{\R^n}u_0=0,
\]
a strictly negative initial mass does not by itself determine whether the
solution is global. The detailed spatial profile of \(u_0\), and not merely
its total mass, becomes essential.

This gives the following qualitative picture:
\[
\begin{array}{c|c|c}
\text{Initial mass} & \text{Range of }p & \text{Typical behavior}\\
\hline
\displaystyle \int u_0>0
& 1<p\le p_F
& \text{non-global}\\[1mm]
\displaystyle \int u_0=0
& 1<p\le p_F
& \text{non-global; sharp lifespan as in Theorem~\ref{thm:main}}\\[1mm]
\displaystyle \int u_0=0
& p>p_F
& \text{small-data global existence}\\[1mm]
\displaystyle \int u_0<0
& p>1
& \text{both global and non-global behavior are possible.}
\end{array}
\]

The last line, due to Pinsky \cite{Pinsky2005}, shows that zero initial mass
is in a certain sense the borderline case in which a sharp quantitative
classification can still be obtained from the interaction between the
linear cancellation and the positive mass generated by the source.

\section{Subcritical regime: proof of the upper lifespan estimates}
\subsection{Mild solutions and nonlinear mass generation}

The mild formulation of \eqref{eq:problem} is
\begin{equation}
u(t)
=
\eps e^{t\Delta}u_0
+
\int_0^t e^{(t-s)\Delta}|u(s)|^p\,ds.
\label{eq:mild}
\end{equation}

\begin{lemma}[Positivity relative to the linear flow]
\label{lem:domination}
Let
\[
v(t):=e^{t\Delta}u_0.
\]
As long as the mild solution exists,
\begin{equation}
u(t,x)\ge \eps v(t,x)
\qquad (t>0,\ x\in\R^n).
\label{eq:domination}
\end{equation}
Consequently,
\begin{equation}
|u(t,x)|^p
\ge
\eps^p\bigl(v(t,x)_+\bigr)^p.
\label{eq:source-lower}
\end{equation}
\end{lemma}

\begin{proof}
The heat semigroup preserves nonnegativity, and $|u|^p\ge0$. Therefore the
Duhamel term in \eqref{eq:mild} is pointwise nonnegative, which gives
\eqref{eq:domination}. If $v(t,x)>0$, then
$u(t,x)\ge\eps v(t,x)>0$, hence
$|u(t,x)|^p\ge\eps^p v(t,x)^p$. If $v(t,x)\le0$, the right-hand side of
\eqref{eq:source-lower} vanishes. This proves the claim.
\end{proof}
For completeness we record the mass identity, which explains the terminology
``generated mass.''

\begin{proposition}[Generation of zeroth-order mass]
\label{prop:mass}
Under \eqref{eq:data-space}, the mild solution belongs to
$L^1(\R^n)\cap L^\infty(\R^n)$ on every compact subinterval of
$[0,T_\eps)$, and
\begin{equation}
\int_{\R^n}u(t,x)\,dx
=
\eps\int_{\R^n}u_0(x)\,dx
+
\int_0^t\int_{\R^n}|u(s,x)|^p\,dx\,ds.
\label{eq:mass-identity}
\end{equation}
In particular, if \eqref{eq:moments} holds and $u_0\not\equiv0$, then
\[
\int_{\R^n}u(t,x)\,dx>0
\qquad\text{for every }t>0.
\]
\end{proposition}

\begin{proof}
Fix $0<T<T_\eps$ and let
\[
M_T:=\sup_{0\le t\le T}\|u(t)\|_\infty<\infty.
\]
The standard fixed-point construction can be performed in
$L^1\cap L^\infty$. Alternatively, once \eqref{eq:mild} is known, one obtains
formally and then by approximation
\begin{align*}
\|u(t)\|_1
&\le
\eps\|u_0\|_1
+
\int_0^t \||u(s)|^p\|_1\,ds
\\&\le
\eps\|u_0\|_1
+
M_T^{p-1}\int_0^t\|u(s)\|_1\,ds.\end{align*}
Gronwall's inequality gives local $L^1$ control. Integrating
\eqref{eq:mild} in space and using
\[
\int_{\R^n}e^{t\Delta}f\,dx=\int_{\R^n}f\,dx
\]
gives \eqref{eq:mass-identity}. If the initial mass is zero, then the right
side is the time integral of a nonnegative function. Since the solution is
nontrivial for small positive times, this integral is strictly positive.
\end{proof}

\subsection{A backward Gaussian blow-up criterion}

The following elementary lemma is the quantitative core of the upper-bound
argument.

\begin{lemma}[Backward Gaussian criterion]
\label{lem:gaussian}
Let $0<T<T_\eps$, let $z\in\R^n$, and define
\[
\Phi(t,x)=G(T-t,x-z),
\qquad
F(t)=\int_{\R^n}u(t,x)\Phi(t,x)\,dx,
\qquad 0\le t<T.
\]
Then
\begin{equation}
F'(t)
=
\int_{\R^n}|u(t,x)|^p\Phi(t,x)\,dx
\ge |F(t)|^p.
\label{eq:Fdiff}
\end{equation}
Hence, if $0\le t_0<T$ and $F(t_0)>0$, then
\begin{equation}
F(t_0)
\le
\bigl((p-1)(T-t_0)\bigr)^{-1/(p-1)}.
\label{eq:gaussian-necessary}
\end{equation}
\end{lemma}

\begin{proof}
Since
\[
\Phi_t+\Delta\Phi=0,
\]
integration by parts gives, for $0<t<T$,
\[
\begin{aligned}
F'(t)
&=
\int_{\R^n}u_t\Phi\,dx+\int_{\R^n}u\Phi_t\,dx\\
&=
\int_{\R^n}(\Delta u+|u|^p)\Phi\,dx
-\int_{\R^n}u\Delta\Phi\,dx\\
&=
\int_{\R^n}|u|^p\Phi\,dx.
\end{aligned}
\]
The identity can be justified first on $[\delta,T-\delta]$ and then extended
to $t=0$ by continuity of the mild solution.

Since $\Phi\ge0$ and
\[
\int_{\R^n}\Phi(t,x)\,dx=1,
\]
Jensen's inequality yields
\[
\int_{\R^n}|u|^p\Phi\,dx
\ge
\left|\int_{\R^n}u\Phi\,dx\right|^p
=
|F(t)|^p.
\]
If $F(t_0)>0$, then $F$ remains positive and
\[
\frac{d}{dt}F^{1-p}
=
(1-p)F^{-p}F'
\le -(p-1).
\]
Integrating from $t_0$ to $t<T$ gives
\[
F(t)^{1-p}
\le
F(t_0)^{1-p}-(p-1)(t-t_0).
\]
The right-hand side must stay nonnegative before $T$, so letting $t\uparrow T$
gives \eqref{eq:gaussian-necessary}.
\end{proof}

\subsection{First-moment asymptotics of the linear flow}

We now make the dipole expansion precise.

\begin{lemma}[Self-similar first-moment expansion]
\label{lem:dipole}
Assume \eqref{eq:data-space}--\eqref{eq:moments}, and put
$v(t)=e^{t\Delta}u_0$. Then
\begin{equation}
t^{(n+1)/2}v(t,\sqrt t\,\xi)
\longrightarrow
-m\cdot\nabla G(1,\xi)
=
\frac{m\cdot\xi}{2}G(1,\xi)
\label{eq:selfsimilar-limit}
\end{equation}
locally uniformly for $\xi\in\R^n$ as $t\to\infty$.
\end{lemma}

\begin{proof}
We have
\[
v(t,\sqrt t\,\xi)
=
t^{-n/2}
\int_{\R^n}
G\left(1,\xi-\frac{y}{\sqrt t}\right)u_0(y)\,dy.
\]
Because $\int u_0=0$,
\[
t^{(n+1)/2}v(t,\sqrt t\,\xi)
=
\int_{\R^n}
\sqrt t
\left[
G\left(1,\xi-\frac{y}{\sqrt t}\right)-G(1,\xi)
\right]u_0(y)\,dy.
\]
For each fixed $(\xi,y)$,
\[
\sqrt t
\left[
G\left(1,\xi-\frac{y}{\sqrt t}\right)-G(1,\xi)
\right]
\longrightarrow
-y\cdot\nabla G(1,\xi).
\]
Moreover, by the mean-value theorem,
\[
\left|
\sqrt t
\left[
G\left(1,\xi-\frac{y}{\sqrt t}\right)-G(1,\xi)
\right]
\right|
\le
|y|\,\|\nabla G(1)\|_\infty.
\]
The majorant is integrable against $|u_0(y)|\,dy$ by
\eqref{eq:data-space}. Dominated convergence therefore gives
\[
t^{(n+1)/2}v(t,\sqrt t\,\xi)
\to
-\int_{\R^n}y\,u_0(y)\,dy\cdot\nabla G(1,\xi)
=
-m\cdot\nabla G(1,\xi).
\]
The same estimate together with uniform continuity of $\nabla G$ yields local
uniformity in $\xi$.
\end{proof}

\begin{corollary}[A positive dipole point]
\label{cor:dipole-point}
Let
\[
e=\frac{m}{|m|}.
\]
For every fixed $a>0$, there exist $c_a>0$ and $T_a>0$ such that
\begin{equation}
v(T,a\sqrt T\,e)
\ge
c_a T^{-(n+1)/2},
\qquad T\ge T_a.
\label{eq:dipole-point}
\end{equation}
\end{corollary}

\begin{proof}
In \eqref{eq:selfsimilar-limit}, take $\xi=ae$. Then
\[
\frac{m\cdot(ae)}2G(1,ae)
=
\frac{a|m|}{2}G(1,ae)>0.
\]
The conclusion follows for all sufficiently large $T$.
\end{proof}

The next localized form is the one needed at the transition exponent.

\begin{lemma}[Positive self-similar sector]
\label{lem:sector}
There exist a compact set
\[
K\Subset\{\xi\in\R^n:m\cdot\xi>0\},
\qquad |K|>0,
\]
and constants $c>0$, $s_0>0$ such that
\begin{equation}
v(s,\sqrt s\,\xi)
\ge
c\,s^{-(n+1)/2},
\qquad
s\ge s_0,\quad \xi\in K.
\label{eq:sector-pointwise}
\end{equation}
Consequently,
\begin{equation}
\int_{\sqrt s K}\bigl(v(s,x)_+\bigr)^p\,dx
\ge
c\,s^{-a},
\qquad
a:=\frac{p(n+1)-n}{2}.
\label{eq:sector-Lp}
\end{equation}
\end{lemma}

\begin{proof}
Choose any compact set of positive measure strictly contained in the half-space
$\{m\cdot\xi>0\}$. The limit in \eqref{eq:selfsimilar-limit} is strictly
positive on a sufficiently small such compact set. By local uniform
convergence, \eqref{eq:sector-pointwise} follows. Changing variables
$x=\sqrt s\,\xi$ then gives
\[
\begin{aligned}
\int_{\sqrt s K}(v_+)^p\,dx
&\ge
c\,s^{-p(n+1)/2}s^{n/2}|K|\\
&=
c\,s^{-[p(n+1)-n]/2}.
\end{aligned}
\]
\end{proof}

\begin{remark}[Origin of $p_m$]
From \eqref{eq:sector-Lp},
\[
\int^{T}s^{-a}\,ds
\]
is algebraically divergent if $a<1$, logarithmically divergent if $a=1$,
and convergent if $a>1$. Since
\[
a=1
\quad\Longleftrightarrow\quad
p=1+\frac1{n+1},
\]
the exponent $p_m$ is exactly the threshold for the time accumulation of the
nonlinearly generated mass coming from the dipole profile.
\end{remark}

\subsection{The dipole-dominated regime}

We first prove the upper bound when
\[
1<p<p_m.
\]

\begin{proposition}
\label{prop:dipole}
If
\[
1<p<1+\frac1{n+1},
\]
then
\begin{equation}
T_\eps
\le
C\eps^{-\frac{2(p-1)}{2-(n+1)(p-1)}}
\label{eq:dipole-bound}
\end{equation}
for all sufficiently small $\eps>0$.
\end{proposition}

\begin{proof}
Let
\[
A_1
:=
\frac{2(p-1)}{2-(n+1)(p-1)}.
\]
Suppose, to the contrary, that for some large constant $K>1$ and a sequence
$\eps\downarrow0$,
\[
T_\eps>K\eps^{-A_1}.
\]
Set
\[
T=K\eps^{-A_1}.
\]
For small $\eps$, $T$ is large. Choose
\[
z_T=a\sqrt T\,e
\]
as in Corollary \ref{cor:dipole-point}. For the backward Gaussian functional
of Lemma \ref{lem:gaussian},
\[
F(0)
=
\int_{\R^n}\eps u_0(x)G(T,x-z_T)\,dx
=
\eps(e^{T\Delta}u_0)(z_T).
\]
Hence
\begin{equation}
F(0)\ge c\eps T^{-(n+1)/2}.
\label{eq:F0-dipole}
\end{equation}
The necessary condition \eqref{eq:gaussian-necessary} gives
\[
c\eps T^{-(n+1)/2}
\le
C T^{-1/(p-1)}.
\]
Thus
\begin{equation}
\eps
T^{d_1}
\le C,
\qquad
d_1:=\frac1{p-1}-\frac{n+1}{2}.
\label{eq:d1-ineq}
\end{equation}
In the present range $d_1>0$, and
\[
A_1=\frac1{d_1}.
\]
Substituting $T=K\eps^{-A_1}$ into \eqref{eq:d1-ineq} yields
\[
K^{d_1}\le C.
\]
Choosing $K$ larger than the fixed constant on the right gives a contradiction.
Therefore \eqref{eq:dipole-bound} holds.
\end{proof}

\begin{remark}
The raw dipole argument actually works whenever
\[
p<1+\frac2{n+1}.
\]
However, once $p>p_m$, the generated-mass mechanism below gives a strictly
stronger upper bound. The two bounds cross exactly at $p=p_m$.
\end{remark}

\subsection{A fixed positive cylinder for the linear flow}

For the generated-mass regime we need only a very elementary positivity fact.

\begin{lemma}
\label{lem:cylinder}
There exist $0<s_0<s_1<\infty$, a bounded measurable set $E\subset\R^n$
with $|E|>0$, and a constant $c_0>0$ such that
\begin{equation}
v(s,x)\ge c_0,
\qquad
(s,x)\in[s_0,s_1]\times E.
\label{eq:cylinder}
\end{equation}
\end{lemma}

\begin{proof}
Fix any $s_*>0$. The heat semigroup is injective, so
$v(s_*)\not\equiv0$ because $u_0\not\equiv0$. Also,
\[
\int_{\R^n}v(s_*,x)\,dx
=
\int_{\R^n}u_0(x)\,dx
=0.
\]
Since $v(s_*)$ is continuous and integrable, it cannot be everywhere
nonpositive unless it vanishes identically. Hence there exists $x_*\in\R^n$
such that
\[
v(s_*,x_*)>0.
\]
Continuity of $v$ in $(s,x)$ gives a small time interval
$[s_0,s_1]$ around $s_*$ and a bounded ball $E$ around $x_*$ on which
$v\ge c_0>0$.
\end{proof}

\subsection{The generated-mass regime}

We now assume
\[
p_m<p<p_F.
\]

\begin{proposition}
\label{prop:mass-regime}
If
\[
1+\frac1{n+1}<p<1+\frac2n,
\]
then
\begin{equation}
T_\eps
\le
C\eps^{-\frac{2p(p-1)}{2-n(p-1)}}
\label{eq:mass-bound}
\end{equation}
for all sufficiently small $\eps>0$.
\end{proposition}

\begin{proof}
Set
\[
A_2
:=
\frac{2p(p-1)}{2-n(p-1)}.
\]
Assume, for contradiction, that for some large $K>1$ and a sequence
$\eps\downarrow0$,
\[
T_\eps>K\eps^{-A_2}.
\]
Put
\[
T=K\eps^{-A_2}.
\]
We use the backward Gaussian centered at $z=0$ and evaluate it at $t=T/2$.
By the semigroup property and the mild formula,
\begin{align}
F(T/2)
&=
\int_{\R^n}u(T/2,x)G(T/2,x)\,dx
\nonumber\\
&=
\eps(e^{T\Delta}u_0)(0)
+
\int_0^{T/2}
\left(e^{(T-s)\Delta}|u(s)|^p\right)(0)\,ds.
\label{eq:Fhalf}
\end{align}

We estimate the linear term. Since $\int u_0=0$,
\[
(e^{T\Delta}u_0)(0)
=
\int_{\R^n}
\bigl(G(T,-y)-G(T,0)\bigr)u_0(y)\,dy.
\]
Using
\[
\|\nabla G(T)\|_\infty
\le C T^{-(n+1)/2},
\]
we obtain
\begin{equation}
\left|(e^{T\Delta}u_0)(0)\right|
\le
C T^{-(n+1)/2}
\int_{\R^n}|y||u_0(y)|\,dy.
\label{eq:linear-origin}
\end{equation}

For the nonlinear term, use Lemma \ref{lem:cylinder}. For small $\eps$,
$T/2>s_1$. On $[s_0,s_1]\times E$, Lemma \ref{lem:domination} gives
\[
|u(s,x)|^p\ge c\eps^p.
\]
Since $s$ and $x$ remain in a fixed bounded set, for all sufficiently large
$T$,
\[
G(T-s,x)\ge cT^{-n/2},
\qquad (s,x)\in[s_0,s_1]\times E.
\]
Therefore
\begin{equation}
\int_0^{T/2}
\left(e^{(T-s)\Delta}|u(s)|^p\right)(0)\,ds
\ge
c\eps^pT^{-n/2}.
\label{eq:fixed-mass}
\end{equation}
Combining \eqref{eq:Fhalf}, \eqref{eq:linear-origin}, and
\eqref{eq:fixed-mass},
\begin{equation}
F(T/2)
\ge
c\eps^pT^{-n/2}
-
C\eps T^{-(n+1)/2}.
\label{eq:Fhalf-two}
\end{equation}

The ratio of the error term to the positive term is bounded by
\[
C\eps^{1-p}T^{-1/2}
=
CK^{-1/2}
\eps^{A_2/2-(p-1)}.
\]
A direct calculation gives
\begin{equation}
\frac{A_2}{2}-(p-1)
=
\frac{(p-1)\bigl((n+1)(p-1)-1\bigr)}
{2-n(p-1)}.
\label{eq:crossover-algebra}
\end{equation}
The denominator is positive because $p<p_F$, while the numerator is positive
precisely because $p>p_m$. Thus the ratio tends to zero as $\eps\downarrow0$.
Consequently,
\begin{equation}
F(T/2)\ge c\eps^pT^{-n/2}
\label{eq:Fhalf-positive}
\end{equation}
for small $\eps$.

Applying Lemma \ref{lem:gaussian} with $t_0=T/2$ gives
\[
c\eps^pT^{-n/2}
\le
C T^{-1/(p-1)}.
\]
Hence
\begin{equation}
\eps^p T^{d_2}\le C,
\qquad
d_2:=\frac1{p-1}-\frac n2>0.
\label{eq:d2-ineq}
\end{equation}
Since
\[
A_2=\frac{p}{d_2},
\]
substitution of $T=K\eps^{-A_2}$ yields
\[
K^{d_2}\le C,
\]
which is impossible if $K$ is chosen sufficiently large. This proves
\eqref{eq:mass-bound}.
\end{proof}

\subsection{The logarithmic transition}

We now treat
\[
p=p_m=1+\frac1{n+1}.
\]
At this exponent the fixed-cylinder estimate is not enough to separate the
positive nonlinear contribution from the dipole error on the common power
scale. Instead we accumulate the source over the expanding self-similar
positive sector of Lemma \ref{lem:sector}.

\begin{proposition}
\label{prop:borderline}
Let
\[
p=1+\frac1{n+1}.
\]
Then
\begin{equation}
T_\eps
\le
C
\eps^{-2/(n+1)}
\left(\log\frac1\eps\right)^{-2/(n+2)}
\label{eq:borderline-bound}
\end{equation}
for all sufficiently small $\eps>0$.
\end{proposition}

\begin{proof}
Let $K_0\subset\R^n$ and $s_0>0$ be given by Lemma \ref{lem:sector}.
At the present exponent,
\[
a=\frac{p(n+1)-n}{2}=1,
\]
and hence
\begin{equation}
\int_{\sqrt s K_0}(v(s,x)_+)^p\,dx
\ge \frac{c}{s},
\qquad s\ge s_0.
\label{eq:critical-sector}
\end{equation}

Let $T<T_\eps$ be large and again use the backward Gaussian centered at the
origin. From \eqref{eq:Fhalf},
\[
F(T/2)
=
\eps(e^{T\Delta}u_0)(0)
+
\int_0^{T/2}\int_{\R^n}
G(T-s,x)|u(s,x)|^p\,dx\,ds.
\]
The linear contribution satisfies
\[
\eps(e^{T\Delta}u_0)(0)
\ge
-C\eps T^{-(n+1)/2}.
\]
For $s_0\le s\le T/2$ and $x\in\sqrt sK_0$, boundedness of $K_0$ gives
\[
|x|\le C\sqrt s\le C\sqrt T,
\]
and $T-s\asymp T$. Hence
\begin{equation}
G(T-s,x)\ge cT^{-n/2}.
\label{eq:kernel-lower}
\end{equation}
Using \eqref{eq:source-lower}, \eqref{eq:critical-sector}, and
\eqref{eq:kernel-lower},
\begin{align}
\int_0^{T/2}\int_{\R^n}
G(T-s,x)|u(s,x)|^p\,dx\,ds
&\ge
c\eps^pT^{-n/2}
\int_{s_0}^{T/2}\frac{ds}{s}
\nonumber\\
&\ge
c\eps^pT^{-n/2}\log T
\label{eq:log-generation}
\end{align}
for all sufficiently large $T$. Thus
\begin{equation}
F(T/2)
\ge
c\eps^pT^{-n/2}\log T
-
C\eps T^{-(n+1)/2}.
\label{eq:F-borderline}
\end{equation}

We now choose the candidate scale
\begin{equation}
T
=
K
\eps^{-2/(n+1)}
\left(\log\frac1\eps\right)^{-2/(n+2)}
\label{eq:T-borderline-choice}
\end{equation}
with $K>1$ fixed and large. The ratio of the positive term in
\eqref{eq:F-borderline} to the magnitude of the linear error is
\[
\eps^{p-1}T^{1/2}\log T.
\]
Since $p-1=1/(n+1)$, \eqref{eq:T-borderline-choice} yields
\[
\eps^{p-1}T^{1/2}\log T
\asymp
K^{1/2}
\left(\log\frac1\eps\right)^{(n+1)/(n+2)}
\longrightarrow\infty.
\]
Therefore, for sufficiently small $\eps$,
\begin{equation}
F(T/2)
\ge
c\eps^pT^{-n/2}\log T.
\label{eq:F-borderline-positive}
\end{equation}

If $T<T_\eps$, Lemma \ref{lem:gaussian} implies
\[
c\eps^pT^{-n/2}\log T
\le
C T^{-1/(p-1)}.
\]
Now
\[
\frac1{p-1}=n+1,
\]
so
\begin{equation}
\eps^pT^{(n+2)/2}\log T\le C.
\label{eq:borderline-necessary}
\end{equation}
For $T$ given by \eqref{eq:T-borderline-choice},
\[
\eps^pT^{(n+2)/2}\log T
\asymp
K^{(n+2)/2}
\]
because $p=(n+2)/(n+1)$ and
$\log T\asymp\log(1/\eps)$. Choosing $K$ sufficiently large contradicts
\eqref{eq:borderline-necessary}. Hence $T_\eps$ cannot exceed a constant
multiple of the scale in \eqref{eq:T-borderline-choice}, proving
\eqref{eq:borderline-bound}.
\end{proof}

\subsection{Upper half of Theorem \ref{thm:main} and comparison of mechanisms}

The upper estimate in Theorem \ref{thm:main} follows directly from Propositions
\ref{prop:dipole}, \ref{prop:borderline}, and \ref{prop:mass-regime}.

We record the comparison between the two algebraic mechanisms. Define
\[
A_{\mathrm d}(p)
=
\frac{2(p-1)}{2-(n+1)(p-1)}
\]
whenever the denominator is positive, and
\[
A_{\mathrm m}(p)
=
\frac{2p(p-1)}{2-n(p-1)}
\]
for $p<p_F$. Solving
\[
A_{\mathrm d}(p)=A_{\mathrm m}(p)
\]
gives
\[
p=1+\frac1{n+1}=p_m.
\]
Moreover,
\[
A_{\mathrm d}(p)<A_{\mathrm m}(p)
\quad\text{for }1<p<p_m,
\]
while
\[
A_{\mathrm m}(p)<A_{\mathrm d}(p)
\quad\text{for }p_m<p<1+\frac2{n+1}.
\]
Since an upper estimate of the form $T_\eps\lesssim\eps^{-A}$ is stronger
when $A$ is smaller, the dipole mechanism is dominant below $p_m$ and the
generated-mass mechanism is dominant above $p_m$.

This calculation explains why $p_m$ is a crossover exponent rather than a
new Fujita exponent. The genuine Fujita threshold for the positive source
remains
\[
p_F=1+\frac2n.
\]

\section{Subcritical regime: proof of the lower lifespan estimates}
\subsection{Cancellation estimates for the linear heat flow}
\label{sec:linear-cancellation}

The lower lifespan estimates require a norm in which the zero-mass cancellation
is visible. The following elementary estimates will be used repeatedly.

\begin{lemma}[Zero-mass heat estimates]
\label{lem:zero-mass-decay}
Assume $u_0\in L^1(\R^n)\cap L^\infty(\R^n)$,
$|x|u_0\in L^1(\R^n)$ and $\int u_0=0$. Let
\[
v(t)=e^{t\Delta}u_0.
\]
Then there exists $C>0$ such that for all $t\ge0$,
\begin{equation}
\|v(t)\|_1\le C(1+t)^{-1/2},
\qquad
\|v(t)\|_\infty\le C(1+t)^{-(n+1)/2}.
\label{eq:linear-L1-Linf}
\end{equation}
More generally, for $1\le q\le\infty$ and $t\ge1$,
\begin{equation}
\|v(t)\|_q
\le
C t^{-\frac12-\frac n2(1-1/q)}\,\||x|u_0\|_1.
\label{eq:linear-Lq-cancel}
\end{equation}
\end{lemma}

\begin{proof}
Since $\int u_0=0$,
\[
v(t,x)=\int_{\R^n}\bigl(G(t,x-y)-G(t,x)\bigr)u_0(y)\,dy.
\]
By the fundamental theorem of calculus,
\[
G(t,x-y)-G(t,x)
=-\int_0^1 y\cdot\nabla G(t,x-\theta y)\,d\theta.
\]
Minkowski's inequality and translation invariance therefore give
\[
\|v(t)\|_q
\le
\||x|u_0\|_1\,\|\nabla G(t)\|_q.
\]
The Gaussian scaling yields
\[
\|\nabla G(t)\|_q
=C_q t^{-\frac12-\frac n2(1-1/q)},
\]
which proves \eqref{eq:linear-Lq-cancel}. For $0<t\le1$ we combine the
semigroup contractions
\[
\|e^{t\Delta}u_0\|_1\le\|u_0\|_1,
\qquad
\|e^{t\Delta}u_0\|_\infty\le\|u_0\|_\infty
\]
with the large-time estimates. This gives \eqref{eq:linear-L1-Linf}.
\end{proof}

\subsection{A unified nonlinear bootstrap for lower lifespan estimates}
\label{sec:lower-bootstrap}

Write the mild solution as
\begin{equation}
u(t)=\eps v(t)+w(t),
\qquad
w(t)=\int_0^t e^{(t-s)\Delta}|u(s)|^p\,ds.
\label{eq:lower-decomp}
\end{equation}
Because the heat kernel and the source are nonnegative,
\[
w(t,x)\ge0.
\]
Define
\begin{equation}
A(t):=\|w(t)\|_1
=\int_0^t\|u(s)\|_p^p\,ds
\label{eq:A-def}
\end{equation}
and
\begin{equation}
H(t):=\eps(1+t)^{-1/2}+A(t).
\label{eq:H-def}
\end{equation}
The equality in \eqref{eq:A-def} follows from Tonelli's theorem and preservation
of mass by the heat semigroup. In particular, $A$ is nondecreasing and
\begin{equation}
A'(t)=\|u(t)\|_p^p
\quad\text{for a.e. }t<T_\eps.
\label{eq:A-derivative}
\end{equation}
Lemma \ref{lem:zero-mass-decay} immediately implies
\begin{equation}
\|u(t)\|_1\le C H(t).
\label{eq:L1-H}
\end{equation}
The $L^\infty$ estimate requires a bootstrap because the Duhamel term has a
nonzero mass.

\begin{lemma}[Bootstrap $L^\infty$ estimate]
\label{lem:Linfty-bootstrap}
There exist constants $\eta>0$ and $C>0$, depending only on $n,p,u_0$, with
the following property. Let $0<S<T_\eps$ and suppose
\begin{equation}
\sup_{0\le t\le S}
H(t)^{p-1}(1+t)^{1-\frac n2(p-1)}
\le\eta.
\label{eq:Q-small}
\end{equation}
Then
\begin{equation}
\|u(t)\|_\infty
\le
C(1+t)^{-n/2}H(t),
\qquad 0\le t\le S.
\label{eq:Linfty-H}
\end{equation}
Consequently,
\begin{equation}
A'(t)
\le
C(1+t)^{-\frac n2(p-1)}
\left[\eps(1+t)^{-1/2}+A(t)\right]^p
\label{eq:A-master}
\end{equation}
for almost every $t\in(0,S)$.
\end{lemma}

\begin{proof}
For $0\le t\le2$, the standard local fixed-point estimate in $L^\infty$ gives
$\|u(t)\|_\infty\le C\eps$ for sufficiently small $\eps$; since
$H(t)\ge c\eps$ and $1+t\asymp1$, this is consistent with
\eqref{eq:Linfty-H}. We therefore concentrate on $t\ge2$.

Split the nonlinear term in \eqref{eq:lower-decomp} as
\[
w(t)=w_{\mathrm{old}}(t)+w_{\mathrm{rec}}(t),
\]
where the time integrals are over $[0,t/2]$ and $[t/2,t]$, respectively.
The $L^1$--$L^\infty$ heat estimate gives
\begin{align}
\|w_{\mathrm{old}}(t)\|_\infty
&\le
C\int_0^{t/2}(t-s)^{-n/2}\|u(s)\|_p^p\,ds\notag\\
&\le C t^{-n/2}A(t).
\label{eq:old-Duhamel}
\end{align}

To treat the recent part, introduce the bootstrap quotient
\[
X(S):=
\sup_{0\le t\le S}
\frac{(1+t)^{n/2}\|u(t)\|_\infty}{H(t)}.
\]
For $s\in[t/2,t]$, monotonicity of $A$ and comparability of $1+s$ and $1+t$
give
\[
H(s)\le C H(t).
\]
Using the $L^\infty$ contraction of the heat semigroup, we obtain
\begin{align}
\|w_{\mathrm{rec}}(t)\|_\infty
&\le
\int_{t/2}^t\|u(s)\|_\infty^p\,ds\notag\\
&\le
C X(S)^p H(t)^p
\int_{t/2}^t(1+s)^{-np/2}\,ds\notag\\
&\le
C X(S)^p H(t)^p(1+t)^{1-np/2}\notag\\
&=
C X(S)^p(1+t)^{-n/2}H(t)
\left[H(t)^{p-1}(1+t)^{1-\frac n2(p-1)}\right].
\label{eq:recent-Duhamel}
\end{align}
The linear contribution satisfies, by Lemma \ref{lem:zero-mass-decay},
\[
\eps\|v(t)\|_\infty
\le
C(1+t)^{-n/2}H(t).
\]
Combining this with \eqref{eq:old-Duhamel}, \eqref{eq:recent-Duhamel} and
\eqref{eq:Q-small} yields
\[
X(S)\le C_0+C_1\eta X(S)^p.
\]
Choose $\eta>0$ so small that the usual continuity argument closes at, say,
$X(S)\le2C_0$. This proves \eqref{eq:Linfty-H}.

Finally, interpolation between $L^1$ and $L^\infty$ gives
\[
A'(t)=\|u(t)\|_p^p
\le
\|u(t)\|_\infty^{p-1}\|u(t)\|_1.
\]
Using \eqref{eq:L1-H} and \eqref{eq:Linfty-H} proves
\eqref{eq:A-master}.
\end{proof}

\begin{remark}
Lemma \ref{lem:Linfty-bootstrap} is the central lower-bound estimate. It
separates the cancellation-bearing linear contribution
$\eps(1+t)^{-1/2}$ in $L^1$ from the nonlinear generated mass $A(t)$. A
crude estimate based on $\eps\|u_0\|_1+A(t)$ would lose the first-moment
improvement and would not recover the dipole-dominated lifespan.
\end{remark}

\subsection{Proof of the lower half of Theorem \ref{thm:main}}
\label{sec:lower-subcritical}

Set
\[
r:=p-1,
\qquad
a:=\frac{p(n+1)-n}{2}
=\frac{1+(n+1)r}{2}.
\]
We prove the three regimes separately. In each case the argument is first
carried out on $[0,S]$ with
$S<\min\{T,T_\eps\}$. The resulting bounds are uniform in $S$. If
$T_\eps\le T$, the $L^\infty$ blow-up alternative then extends the solution
past $T_\eps$, a contradiction. Hence $T_\eps>T$.

\paragraph{The dipole-dominated range $1<p<p_m$}

Here
\[
r<\frac1{n+1},
\qquad
a<1.
\]
Define
\[
\theta:=1-\frac{n+1}{2}r>0,
\qquad
\gamma:=1-a=\frac{1-(n+1)r}{2}>0.
\]
Fix $K\gg1$ and bootstrap
\begin{equation}
A(t)\le K\eps^p(1+t)^\gamma.
\label{eq:A-bootstrap-low}
\end{equation}
Choose
\begin{equation}
T=c_0\eps^{-r/\theta}
=c_0\eps^{-\frac{2(p-1)}{2-(n+1)(p-1)}},
\label{eq:T-lower-low}
\end{equation}
where $c_0>0$ will be chosen small. From \eqref{eq:A-bootstrap-low},
\begin{align*}
\frac{A(t)}{\eps(1+t)^{-1/2}}
&\le
K\eps^r(1+t)^{\gamma+1/2}\\
&=K\eps^r(1+t)^\theta
\le C K c_0^\theta,
\qquad t\le T.
\end{align*}
Thus, for $c_0$ sufficiently small,
\begin{equation}
H(t)\le C\eps(1+t)^{-1/2}.
\label{eq:H-linear-low}
\end{equation}
Moreover,
\[
H(t)^r(1+t)^{1-nr/2}
\le
C\eps^r(1+t)^\theta
\le Cc_0^\theta,
\]
so the smallness hypothesis of Lemma \ref{lem:Linfty-bootstrap} holds.
Equation \eqref{eq:A-master} then gives
\begin{align*}
A'(t)
&\le
C\eps^p(1+t)^{-nr/2-p/2}\\
&=C\eps^p(1+t)^{-a}.
\end{align*}
Since $a<1$,
\[
A(t)\le C\eps^p(1+t)^{1-a}
=C\eps^p(1+t)^\gamma.
\]
Taking $K$ larger than the last constant closes the bootstrap. Consequently,
\begin{equation}
T_\eps\ge
c\eps^{-\frac{2(p-1)}{2-(n+1)(p-1)}}.
\label{eq:lower-dipole-final}
\end{equation}

\paragraph{The transition power $p=p_m$}

Let
\[
p=p_m=1+\frac1{n+1},
\qquad r=\frac1{n+1}.
\]
Then $a=1$. Set $L(t):=\log(e+t)$ and bootstrap
\begin{equation}
A(t)\le K\eps^p L(t).
\label{eq:A-bootstrap-border}
\end{equation}
Using $(x+y)^p\le C(x^p+y^p)$ in \eqref{eq:A-master}, we obtain
\begin{equation}
A'(t)
\le
C\eps^p(1+t)^{-1}
+C(1+t)^{-\frac{n}{2(n+1)}}A(t)^p.
\label{eq:A-border-diff}
\end{equation}
Under \eqref{eq:A-bootstrap-border}, integration gives
\begin{equation}
A(t)
\le
C\eps^p L(t)
+C K^p\eps^{p^2}
(1+t)^{\frac{n+2}{2(n+1)}}L(t)^p.
\label{eq:A-border-integrated}
\end{equation}
Thus the second term can be absorbed into $K\eps^pL(t)$ whenever
\begin{equation}
\eps^{p(p-1)}
(1+T)^{\frac{n+2}{2(n+1)}}
L(T)^{p-1}
\le c_1.
\label{eq:border-lower-condition}
\end{equation}
Take
\begin{equation}
T=c_0
\eps^{-\frac2{n+1}}
\left(\log\frac1\eps\right)^{-\frac2{n+2}}.
\label{eq:T-lower-border}
\end{equation}
Since $p=(n+2)/(n+1)$ and $L(T)\asymp\log(1/\eps)$,
condition \eqref{eq:border-lower-condition} holds for small $c_0$.

It remains to verify the smallness condition \eqref{eq:Q-small}. From the
bootstrap,
\begin{align*}
H(t)^r(1+t)^{1-\frac{nr}{2}}
&\le C\eps^r(1+t)^{1/2}
+C\eps^{pr}L(t)^r
(1+t)^{\frac{n+2}{2(n+1)}}.
\end{align*}
At the scale \eqref{eq:T-lower-border}, the first term is
$O((\log(1/\eps))^{-1/(n+2)})$, while the second is made uniformly small by
choosing $c_0$ small, exactly as in \eqref{eq:border-lower-condition}.
Lemma \ref{lem:Linfty-bootstrap} is therefore applicable and the bootstrap
closes. We conclude
\begin{equation}
T_\eps\ge
c\eps^{-\frac2{n+1}}
\left(\log\frac1\eps\right)^{-\frac2{n+2}}.
\label{eq:lower-border-final}
\end{equation}

\paragraph{The generated-mass range $p_m<p<p_F$}

Now
\[
a>1,
\qquad
\beta:=1-\frac n2(p-1)>0.
\]
Bootstrap
\begin{equation}
A(t)\le K\eps^p.
\label{eq:A-bootstrap-mass}
\end{equation}
Again using $(x+y)^p\le C(x^p+y^p)$ in \eqref{eq:A-master}, we find
\[
A'(t)
\le
C\eps^p(1+t)^{-a}
+C(1+t)^{-\frac n2(p-1)}A(t)^p.
\]
Since $a>1$, integration and \eqref{eq:A-bootstrap-mass} yield
\begin{equation}
A(t)
\le
C\eps^p
+C K^p\eps^{p^2}(1+t)^\beta.
\label{eq:A-mass-integrated}
\end{equation}
Choose
\begin{equation}
T=c_0\eps^{-\frac{p(p-1)}{\beta}}
=c_0\eps^{-\frac{2p(p-1)}{2-n(p-1)}}.
\label{eq:T-lower-mass}
\end{equation}
Then
\[
\eps^{p(p-1)}T^\beta=c_0^\beta,
\]
so the second term in \eqref{eq:A-mass-integrated} is absorbed by the
bootstrap for $c_0$ sufficiently small.

We also verify \eqref{eq:Q-small}. The generated-mass part satisfies
\[
(\eps^p)^r(1+T)^\beta
\le Cc_0^\beta.
\]
For the linear part, if
$1-(n+1)r/2\le0$ there is nothing to prove. Otherwise,
\begin{align*}
\eps^r T^{1-(n+1)r/2}
&=c_0^{1-(n+1)r/2}
\eps^{\frac{r^2((n+1)r-1)}{2-nr}},
\end{align*}
which tends to zero because $r>1/(n+1)$. Thus Lemma
\ref{lem:Linfty-bootstrap} applies and the bootstrap closes. Hence
\begin{equation}
T_\eps\ge
c\eps^{-\frac{2p(p-1)}{2-n(p-1)}}.
\label{eq:lower-mass-final}
\end{equation}

Combining \eqref{eq:lower-dipole-final}, \eqref{eq:lower-border-final} and
\eqref{eq:lower-mass-final} with the upper estimates proved in the previous
sections completes the proof of Theorem \ref{thm:main}.

\section{Critical regime: proof of the upper and lower bounds}
\subsection{The Fujita endpoint: the lower bound}
\label{sec:critical-lower}

We first prove the lower estimate in \eqref{eq:sharp-critical}. The argument is
the limiting case of the generated-mass bootstrap, but it is useful to present
it separately because the algebraic time integral becomes logarithmic.

\begin{proof}[Proof of the lower bound in \eqref{eq:sharp-critical}]
Let
\[
p=p_F=1+\frac2n,
\qquad r=p-1=\frac2n.
\]
Then
\[
\frac n2r=1.
\]
Bootstrap, as in the generated-mass regime,
\begin{equation}
A(t)\le K\eps^p.
\label{eq:A-bootstrap-critical}
\end{equation}
The smallness quantity in Lemma \ref{lem:Linfty-bootstrap} now becomes simply
\[
H(t)^r,
\]
because $1-nr/2=0$. Under \eqref{eq:A-bootstrap-critical},
\[
H(t)^r\le C(\eps^r+\eps^{pr}),
\]
which is uniformly small for small $\eps$, independently of time.
Therefore Lemma \ref{lem:Linfty-bootstrap} applies on every interval on which
the bootstrap holds.

The master inequality \eqref{eq:A-master} gives
\begin{align}
A'(t)
&\le
C(1+t)^{-1}
\left[\eps(1+t)^{-1/2}+A(t)\right]^p\notag\\
&\le
C\eps^p(1+t)^{-a_F}
+C(1+t)^{-1}A(t)^p,
\label{eq:A-critical-diff}
\end{align}
where
\[
a_F=\frac{p(n+1)-n}{2}>1.
\]
Hence the first term is integrable. Under
\eqref{eq:A-bootstrap-critical}, integration yields
\begin{equation}
A(t)
\le
C\eps^p
+C K^p\eps^{p^2}\log(e+t).
\label{eq:A-critical-integrated}
\end{equation}
Choose
\begin{equation}
T=\exp\!\left(c_0\eps^{-p(p-1)}\right).
\label{eq:T-critical-lower-choice}
\end{equation}
Then
\[
\eps^{p^2}\log T
=c_0\eps^{p^2-p(p-1)}
=c_0\eps^p.
\]
Choosing $c_0$ sufficiently small and $K$ sufficiently large closes
\eqref{eq:A-bootstrap-critical}. The same continuation argument used in
Section \ref{sec:lower-subcritical} then implies
\[
T_\eps\ge \exp\!\left(c\eps^{-p(p-1)}\right),
\]
which is the lower bound in \eqref{eq:sharp-critical}.
\end{proof}

\subsection{The Fujita endpoint: the critical upper bound}
\label{sec:critical-upper}

We now prove the upper estimate in \eqref{eq:sharp-critical}. The simple
backward-Gaussian ODE used in the strict subcritical range becomes scale
invariant at $p=p_F$, so a critical scale-ODE argument is required. The main
new point is that the datum obtained after shifting the solution to a fixed
positive time is still sign-changing. We therefore first prove a critical
test-function lemma that requires only positive total mass and a finite first
absolute moment.

\paragraph{A critical test-function lemma for sign-changing data}

Let
\[
p=p_F=1+\frac2n,
\qquad
p'=\frac{p}{p-1}=1+\frac n2.
\]
Choose a nonincreasing function $\eta\in C^\infty([0,\infty))$ satisfying
\[
0\le\eta\le1,\qquad
\eta(s)=1\ \text{for }0\le s\le\frac12,
\qquad
\eta(s)=0\ \text{for }s\ge1.
\]
Define the auxiliary cutoff
\[
\eta^*(s)=
\begin{cases}
0,&0\le s<1/2,\\
\eta(s),&s\ge1/2.
\end{cases}
\]
For $R>0$, set
\begin{equation}
\psi_R(t,x)
=
\left[
\eta\!\left(\frac{t+|x|^2}{R}\right)
\right]^{2p'},
\qquad
\psi_R^*(t,x)
=
\left[
\eta^*\!\left(\frac{t+|x|^2}{R}\right)
\right]^{2p'}.
\label{eq:critical-test-functions}
\end{equation}
The parameter $R$ has the dimension of time; the corresponding spatial cutoff
has radius of order $\sqrt R$.

\begin{lemma}[Critical sign-changing test-function lemma]
\label{lem:critical-sign}
Let $f\in L^1(\R^n)\cap L^\infty(\R^n)$ satisfy
\[
|x|f\in L^1(\R^n),
\qquad
M:=\int_{\R^n}f(x)\,dx>0,
\qquad
B:=\int_{\R^n}|x|\,|f(x)|\,dx<\infty.
\]
Let $v$ be a bounded mild solution of
\[
v_t-\Delta v=|v|^p,
\qquad
v(0,x)=f(x),
\qquad p=1+\frac2n,
\]
on $[0,S)$. Then
\begin{equation}
S\le
R_0\exp\!\left(CM^{-(p-1)}\right),
\qquad
R_0:=\max\left\{1,8\left(\frac BM\right)^2\right\},
\label{eq:critical-sign-lemma-bound}
\end{equation}
where $C>0$ depends only on $n$ and the fixed cutoff $\eta$.
\end{lemma}

\begin{proof}
If $S\le R_0$, the conclusion is immediate. We therefore assume $S>R_0$ and
fix $R$ with
\[
R_0<R<S.
\]
Since $\psi_R(0,x)=1$ whenever $|x|\le\sqrt{R/2}$, we have
\begin{align}
\int_{\R^n}f(x)\psi_R(0,x)\,dx
&=M-\int_{\R^n}f(x)(1-\psi_R(0,x))\,dx\notag\\
&\ge M-\int_{|x|\ge\sqrt{R/2}}|f(x)|\,dx\notag\\
&\ge M-\sqrt{\frac2R}\,B.
\label{eq:critical-initial-pairing}
\end{align}
By the definition of $R_0$,
\begin{equation}
\int_{\R^n}f(x)\psi_R(0,x)\,dx\ge\frac M2
\qquad (R\ge R_0).
\label{eq:critical-initial-positive}
\end{equation}

We next estimate the derivatives of $\psi_R$. On the support of the
derivatives of $\psi_R$,
\[
\frac12\le\frac{t+|x|^2}{R}\le1,
\qquad |x|^2\le R.
\]
A direct differentiation of \eqref{eq:critical-test-functions} gives
\begin{equation}
|\partial_t\psi_R+\Delta\psi_R|
\le
\frac CR(\psi_R^*)^{1/p}.
\label{eq:critical-derivative-bound}
\end{equation}
Indeed, the first derivative terms contain powers at least
$\eta^{2p'-1}$, while the second derivative terms contain
$\eta^{2p'-2}$, and
\[
2p'-2=\frac{2p'}p.
\]
The factors $|x|^2/R^2$ arising from the Hessian are bounded by $R^{-1}$ on
the derivative support.

Because $R<S$, the test function vanishes before the terminal time. The weak
formulation therefore gives
\begin{equation}
\iint |v|^p\psi_R\,dx\,dt
+
\int f(x)\psi_R(0,x)\,dx
=
-\iint v(\partial_t\psi_R+\Delta\psi_R)\,dx\,dt.
\label{eq:critical-weak-test}
\end{equation}
Here and below the space-time integrals are over $(0,S)\times\R^n$; the
integrands are supported in $0<t<R$, so no ambiguity arises. Define
\[
I(R):=\iint |v|^p\psi_R\,dx\,dt,
\qquad
y(R):=\iint |v|^p\psi_R^*\,dx\,dt.
\]
Using \eqref{eq:critical-initial-positive},
\eqref{eq:critical-derivative-bound}, H\"older's inequality, and
\[
|\supp\psi_R^*|\le CR^{1+n/2},
\]
we obtain
\begin{align}
I(R)+\frac M2
&\le
CR^{-1}y(R)^{1/p}R^{(1+n/2)/p'}\notag\\
&=Cy(R)^{1/p},
\label{eq:critical-basic-scale}
\end{align}
because $p'=1+n/2$. This is the critical cancellation of the scale factor.

Following the scale-integration device of Ikeda and Sobajima
\cite{IkedaSobajima2019}, define
\begin{equation}
Y(R):=\int_0^R y(r)\,\frac{dr}{r}.
\label{eq:critical-Y}
\end{equation}
Then $Y$ is absolutely continuous and
\[
Y'(R)=\frac{y(R)}R
\]
for almost every $R$. Moreover,
\begin{equation}
Y(R)\le(\log2)I(R).
\label{eq:critical-Y-I}
\end{equation}
To verify this, use Fubini's theorem. For a fixed point $(t,x)$, with
$\sigma=(t+|x|^2)/R$, the inner scale integral equals
\[
\int_\sigma^\infty [\eta^*(s)]^{2p'}\,\frac{ds}{s}.
\]
If $\sigma<1/2$, this is at most $\log2$ and $\eta(\sigma)=1$; if
$1/2\le\sigma\le1$, monotonicity of $\eta$ gives
\[
\int_\sigma^1\eta(s)^{2p'}\frac{ds}{s}
\le
(\log2)\eta(\sigma)^{2p'}.
\]
This proves \eqref{eq:critical-Y-I}.

Set
\[
\delta:=\frac M2,
\qquad
Z(R):=\delta+\frac{Y(R)}{\log2}.
\]
Combining \eqref{eq:critical-basic-scale} and \eqref{eq:critical-Y-I},
\[
Z(R)^p\le Cy(R)=CRY'(R)=C(\log2)RZ'(R)
\]
for almost every $R\in(R_0,S)$. Hence
\[
\frac{Z'(R)}{Z(R)^p}\ge\frac cR.
\]
Since
\[
-\frac{d}{dR}Z(R)^{1-p}
=(p-1)Z(R)^{-p}Z'(R),
\]
we obtain
\[
-\frac{d}{dR}Z(R)^{1-p}\ge\frac cR
\]
for almost every $R\in(R_0,S)$. Integrating from $R_0$ to $R<S$ gives
\[
Z(R)^{1-p}
\le
Z(R_0)^{1-p}-c\log\frac R{R_0}.
\]
The left-hand side is nonnegative. Since $Z(R_0)\ge\delta$ and $1-p<0$,
\[
\log\frac R{R_0}
\le
C\delta^{-(p-1)}
\le
CM^{-(p-1)}.
\]
Letting $R\uparrow S$ proves \eqref{eq:critical-sign-lemma-bound}.
\end{proof}

\begin{remark}
Lemma \ref{lem:critical-sign} is an adaptation of the critical scale-ODE
test-function method in \cite{IkedaSobajima2019}. The essential new point for
our purposes is \eqref{eq:critical-initial-pairing}, which replaces pointwise
nonnegativity of the initial datum by positive total mass plus a finite first
absolute moment.
\end{remark}

\paragraph{Positive mass generated at a fixed time}

We apply Lemma \ref{lem:critical-sign} to a time shift of the zero-mass
solution. The following estimates record precisely what is needed.

\begin{proposition}[Fixed-time generated mass and first moment]
\label{prop:fixed-mass-critical}
Fix $\tau>0$. Assume \eqref{eq:data-space},
$\int_{\R^n}u_0=0$, and $u_0\not\equiv0$. For sufficiently small $\eps$ the
solution exists on $[0,\tau]$, and there are constants $c_\tau,C_\tau>0$
such that
\begin{equation}
c_\tau\eps^p
\le
M_\eps(\tau):=\int_{\R^n}u(\tau,x)\,dx
\le
C_\tau\eps^p,
\label{eq:critical-fixed-mass-two-sided}
\end{equation}
and
\begin{equation}
B_\eps(\tau):=
\int_{\R^n}|x|\,|u(\tau,x)|\,dx
\le C_\tau\eps.
\label{eq:critical-fixed-moment}
\end{equation}
\end{proposition}

\begin{proof}
On the fixed interval $[0,\tau]$, standard local theory and the mild formula
give
\begin{equation}
\sup_{0\le t\le\tau}
\bigl(\|u(t)\|_1+\|u(t)\|_\infty\bigr)
\le C_\tau\eps.
\label{eq:critical-local-small}
\end{equation}
Integrating the equation in space and using $\int u_0=0$ yields the exact
mass identity
\begin{equation}
M_\eps(\tau)
=
\int_0^\tau\|u(s)\|_p^p\,ds.
\label{eq:critical-mass-identity}
\end{equation}
The upper bound in \eqref{eq:critical-fixed-mass-two-sided} follows at once
from \eqref{eq:critical-local-small} and interpolation.

For the lower bound, put $v(s)=e^{s\Delta}u_0$. Since the Duhamel term is
nonnegative,
\[
u(s,x)\ge\eps v(s,x).
\]
On the set where $v(s,x)>0$, this implies
\[
|u(s,x)|^p\ge\eps^p(v(s,x)_+)^p.
\]
Hence
\begin{equation}
M_\eps(\tau)
\ge
\eps^p\int_0^\tau\|(e^{s\Delta}u_0)_+\|_p^p\,ds
=:c_\tau\eps^p.
\label{eq:critical-mass-lower-direct}
\end{equation}
The constant $c_\tau$ is strictly positive. Indeed, the heat semigroup is
injective, so $e^{s\Delta}u_0\not\equiv0$ for $s>0$; its integral is zero,
and therefore its positive part cannot vanish identically.

It remains to prove \eqref{eq:critical-fixed-moment}. We use the elementary
weighted heat estimate
\begin{equation}
\||x|e^{t\Delta}g\|_1
\le
\||x|g\|_1+C\sqrt t\,\|g\|_1,
\label{eq:weighted-heat-critical}
\end{equation}
which follows from $|x|\le|y|+|x-y|$ inside the heat-kernel convolution. Let
\[
W(t):=\int_{\R^n}(1+|x|)|u(t,x)|\,dx.
\]
Applying \eqref{eq:weighted-heat-critical} to the mild formula, using
$t\le\tau$, and estimating
\[
\|(1+|x|)|u|^p\|_1
\le
\|u\|_\infty^{p-1}W(t),
\]
we obtain
\[
W(t)
\le
C_\tau\eps
+C_\tau\int_0^t\|u(s)\|_\infty^{p-1}W(s)\,ds.
\]
By \eqref{eq:critical-local-small} and Gronwall's inequality,
\[
\sup_{0\le t\le\tau}W(t)\le C_\tau\eps,
\]
which proves \eqref{eq:critical-fixed-moment}.
\end{proof}

\paragraph{Completion of the critical upper estimate}

We now complete the endpoint proof. If $T_\eps\le1$, the upper bound in
\eqref{eq:sharp-critical} is trivial for small $\eps$. Assume therefore
$T_\eps>1$ and set
\[
f_\eps(x):=u(1,x).
\]
By Proposition \ref{prop:fixed-mass-critical},
\[
M_\eps:=\int f_\eps\ge c\eps^p,
\qquad
B_\eps:=\int |x||f_\eps(x)|\,dx\le C\eps.
\]
Apply Lemma \ref{lem:critical-sign} to the shifted solution
\[
V(t,x)=u(t+1,x).
\]
Its lifespan is $S=T_\eps-1$, and the starting scale in
\eqref{eq:critical-sign-lemma-bound} satisfies
\begin{equation}
R_{0,\eps}
\le
C\left(\frac{\eps}{\eps^p}\right)^2
=
C\eps^{-2(p-1)}.
\label{eq:critical-R0-eps}
\end{equation}
Moreover,
\[
M_\eps^{-(p-1)}
\le
C\eps^{-p(p-1)}.
\]
Therefore
\begin{align}
T_\eps-1
&\le
C\eps^{-2(p-1)}
\exp\!\left(C\eps^{-p(p-1)}\right)\notag\\
&\le
\exp\!\left(C'\eps^{-p(p-1)}\right)
\label{eq:critical-upper-final}
\end{align}
for sufficiently small $\eps$. This proves the upper bound in
\eqref{eq:sharp-critical}. Together with Section \ref{sec:critical-lower}, the
critical part of Theorem \ref{thm:main} is complete.

\section{Conclusion and further remarks}

We have determined the small-data lifespan for
$$
u_t-\Delta u=|u|^p
$$
throughout the Fujita range
$$
1<p\le p_F,
\qquad
p_F=1+\frac2n,
$$
for zero-mass sign-changing initial data with finite first absolute moment, and with nonzero first moment in the regimes where the dipole upper bound is used. The zero-mass condition removes the leading Gaussian term from the linear evolution and replaces it by a dipole profile. At the same time, the positivity of the source
$
|u|^p\ge0
$
immediately creates a positive zeroth-order mass. The interaction between these two mechanisms produces the additional threshold
$$
p_m=1+\frac1{n+1},
$$
which lies strictly below the Fujita exponent.

The sign of the nonlinearity is essential in this mechanism. For the odd equation
$$
u_t-\Delta u=|u|^{p-1}u,
$$
positive and negative parts of the source may cancel, and zero-mass configurations can be preserved by suitable symmetries; this may lead to global sign-changing solutions in parameter ranges where nonnegative solutions blow up, see \cite{Ghoul2012}. In contrast, for the present equation the source is nonnegative, and therefore
$$
\frac{d}{dt}\int_{\mathbb R^n}u(t,x)\,dx
=
\int_{\mathbb R^n}|u(t,x)|^p\,dx\ge0.
$$
Thus the initial cancellation of the total mass is destroyed immediately unless the solution is identically zero. This explains why the long-time dynamics eventually returns to a zeroth-order Fujita mechanism, even though the initial linear asymptotics are governed by a dipole.

Combining the backward-Gaussian upper estimates with the cancellation-aware
\(L^1\)--\(L^\infty\) bootstrap, we obtain the sharp subcritical lifespan laws
$$
T_\eps\asymp
\begin{cases}
\eps^{-\frac{2(p-1)}{2-(n+1)(p-1)}},
&1<p<p_m,\\[2mm]
\eps^{-2/(n+1)}
\bigl(\log(1/\eps)\bigr)^{-2/(n+2)},
&p=p_m,\\[2mm]
\eps^{-\frac{2p(p-1)}{2-n(p-1)}},
&p_m<p<p_F.
\end{cases}
$$
Hence, below \(p_m\), the lifespan is determined by the initial dipole mode. At \(p=p_m\), the nonlinear mass generated by the dipole accumulates only logarithmically, which produces the logarithmic correction. Above \(p_m\), the generated positive mass dominates and the lifespan is governed by the usual Fujita balance with an effective amplitude of order \(\eps^p\).

At the Fujita endpoint
$
p=p_F,
$
the same effective-mass mechanism leads to an exponential lifespan. After any fixed positive time, the solution has generated a positive mass of order
$
\eps^{p_F}.
$
The cancellation-aware bootstrap gives the lower estimate
$$
T_\eps\ge
\exp\!\left(
c\eps^{-p_F(p_F-1)}
\right),
$$
while the sign-changing critical scale-ODE test-function argument yields
$$
T_\eps\le
\exp\!\left(
C\eps^{-p_F(p_F-1)}
\right).
$$
Thus the zero-mass cancellation modifies the effective small parameter at the critical exponent, replacing the classical Lee-Ni scale \(\eps^{-(p_F-1)}\) by the stronger scale \(\eps^{-p_F(p_F-1)}\).

The endpoint argument also clarifies the role of spatial decay assumptions. In Pinsky's qualitative critical non-globality theorem, an additional Gaussian control of the negative part of the initial datum is imposed. In the present quantitative approach, the finite absolute first moment
$$
\int_{\mathbb R^n}|x|\,|u_0(x)|\,dx<\infty
$$
provides instead a direct estimate of the large-scale cutoff pairing after a fixed positive time. This is sufficient for the critical scale-ODE argument and leads to the sharp exponential upper bound.

Finally, if higher moments of the initial datum also vanish, then the leading linear profile is no longer a dipole but a higher-order derivative of the Gaussian. More precisely, if
$$
\int_{\mathbb R^n}x^\alpha u_0(x)\,dx=0
\qquad (|\alpha|<k),
$$
while at least one moment of order \(k\) is nonzero, then the linear evolution is governed asymptotically by a linear combination of terms
$$
\partial^\alpha G(t,x),
\qquad |\alpha|=k.
$$

For the present positive-source equation, however, a positive zeroth-order mass is still generated immediately by the nonlinearity. Therefore the higher-order cancellation case cannot be obtained by a simple formal replacement of \(n+1\) by \(n+k\) in the lifespan formulas above. The precise competition between the initial higher-order multipole and the nonlinearly generated Gaussian mass remains an interesting problem for further study.

\section*{Declaration of competing interest}
The Author declare that there is no conflict of interest.

\section*{Data Availability Statements}
The manuscript has no associated data.

\section*{Declaration of generative AI and AI-assisted technologies in the manuscript preparation process.}
During the preparation of this work, the author used ChatGPT (OpenAI) to assist with language editing, improvement of readability, and organization of selected parts of the manuscript. All mathematical arguments, proofs, references, interpretations, and conclusions were independently checked and verified by the author, who takes full responsibility for the content of the manuscript.

\section*{Acknowledgment} The author is supported by the Science Committee of the Ministry of Education and Science of the Republic of Kazakhstan (Grant No. BR31714735).

\end{document}